\documentclass[reqno]{amsproc}
\usepackage{amsmath}
\usepackage{amsfonts}
\usepackage{mathrsfs}
\usepackage[dvips]{graphicx}
\usepackage{amssymb}
\usepackage{amsbsy}
\usepackage{amsthm}
\usepackage{graphics}
\usepackage{color}
\usepackage{textcomp}
\usepackage{tikz,tikz-cd}
\usetikzlibrary{matrix,arrows}
 \makeatletter
\tikzset{
  edge node/.code={%
      \expandafter\def\expandafter\tikz@tonodes\expandafter{\tikz@tonodes
      #1}}} \makeatother \tikzset{
      subseteq/.style={ draw=none, edge
      node={node [sloped, allow upside down,
      auto=false]{$\subsetneq$}}},
      Subseteq/.style={ draw=none, every
      to/.append style={ edge node={node
      [sloped, allow upside down,
      auto=false]{$\subsetneq$}}}}}

\numberwithin{equation}{section}

\newtheorem{theorem}{Theorem}[section]
\newtheorem{corollary}[theorem]{Corollary}

\newtheorem{prob}[theorem]{Problem}
\newtheorem{lemma}[theorem]{Lemma}

\theoremstyle{definition}

\theoremstyle{remark}
\newtheorem{rem}[theorem]{Remark}

\newcommand*{\hh}{\mathscr{H}}

\newcommand*{\ascr}{\mathscr{A}}
\newcommand*{\ddd}{\mathscr{D}}
\newcommand*{\natu}{\mathbb{N}}

\newcommand*{\nul}{\mathscr{N}}
\newcommand*{\real}{\mathbb{R}}
\newcommand*{\comp}{\mathbb{C}}
\newcommand*{\borel}{\mathfrak{B}}

\newcommand*{\cbb}{\comp}
\newcommand*{\D}{\mathrm{d\hspace{.1ex}}}

\newcommand*{\jd}[1]{{\mathscr N}(#1)}
\newcommand*{\Le}{\leqslant}
\newcommand*{\lambdab}{\boldsymbol\lambda}
\newcommand*{\Ge}{\geqslant}
\newcommand*{\ran}{\mathscr{R}}

\newcommand*{\ogr}{\boldsymbol{B}}
\newcommand*{\kk}{\mathscr{K}}
\newcommand*{\mm}{\mathscr{M}}
\newcommand*{\rbb}{\mathbb{R}}
\newcommand*{\zbb}{\mathbb{Z}}
   
\begin{document}

\title[On $n$th roots of quasinormal
operators]{On $n$th roots of bounded and unbounded\\
quasinormal operators}

   \author[P. Pietrzycki and J. Stochel]{Pawe{\l} Pietrzycki and  Jan Stochel}


   \thanks{This research of both authors was supported by the Priority Research Area
SciMat under the program Excellence
Initiative-Research University at the Jagiellonian
University in Krakow, Poland.}


   \address{Wydzia{\l} Matematyki i Informatyki, Uniwersytet
Jagiello\'{n}ski, ul. {\L}ojasiewicza 6, PL-30348
Krak\'{o}w, Poland.}

   \email{pawel.pietrzycki@im.uj.edu.pl}

   \address{Wydzia{\l} Matematyki i Informatyki, Uniwersytet
Jagiello\'{n}ski, ul. {\L}ojasiewicza 6, PL-30348
Krak\'{o}w, Poland.}

   \email{jan.stochel@im.uj.edu.pl}

   \begin{abstract}
In a recent paper \cite{Curto20}, R. E. Curto,
S. H. Lee and J. Yoon asked the following
question: {\em Let $T$ be a subnormal operator,
and assume that $T^2$ is quasinormal. Does it
follow that $T$ is quasinormal?}. In
\cite{P-S21} we answered this question in the
affirmative. In the present paper, we will
extend this result in two directions. Namely, we
prove that both class A $n$th roots of bounded
quasinormal operators and subnormal $n$th roots
of unbounded quasinormal operators are
quasinormal. We also show that a non-normal
quasinormal operator having a quasinormal $n$th
root has a non-quasinormal $n$th root.
   \end{abstract}
   \maketitle
   \section{\label{Sec.1}Introduction}
The importance of the spectral theorem in
mathematics and its applications was a
motivation for the search for wider classes of
operators inheriting some properties of the
ancestors. Consequently, there have been many
generalizations obtained by weakening the
conditions defining normal operators. Let us
recall some of them that are the subject of our
research in this article.

Denote by $\ogr(\hh)$ the $C^*$-algebra of all
bounded linear operators on a complex Hilbert
space $\hh$ and by $I=I_\hh$ the identity
operator on $\hh$. We write $\ogr_+(\hh)$ for
the convex cone of all positive selfadjoint
elements of $\ogr(\hh)$. Given another complex
Hilbert spaces $\kk$, we denote by
$\boldsymbol{B}(\hh,\kk)$ the Banach space of
all bounded linear operators from $\hh$ to
$\kk$. The kernel and the range of $T\in
\ogr(\hh,\kk)$ are denoted by $\jd{T}$ and
$\ran(T)$, respectively.

An operator $T\in \ogr(\hh)$ is said to be:
   \begin{itemize}
   \item \textit{normal} if $T^*T=TT^*$,
   \item \textit{quasinormal} if $T(T^*T)=(T^*T)T$,
   \item \textit{subnormal} if $T$ is (unitarily equivalent to) the restriction
of a normal operator to its closed invariant subspace,
   \item \textit{hyponormal} if $TT^* \Le T^*T$,
   \item \textit{$p$-hyponormal} if $(TT^*)^p \Le (T^*T)^p$, where $p$ is
a positive real number,
   \item \textit{log-hyponormal} if $T$ is invertible
in $\ogr(\hh)$ and $\log TT^* \Le \log T^*T$,
   \item of \textit{class A} if $T^*T \Le (T^{*2}T^2)^{\frac12}$,
   \item \textit{paranormal} if $\|Th\|^2 \Le
\|T^2h\|\|h\|$ for all $h\in \hh$.
  \end{itemize}
The structure of the inclusion relations between the classes
of operators defined above is illustrated in
Figure~\ref{inkluzje} (for more information, see
Section~\ref{Sec.2}). The classes of subnormal and hyponormal
operators were introduced by P. R. Halmos in \cite{hal50}.
The study of quasinormal operators was initiated by A. Brown
in \cite{brow53}. In turn, the notions of a paranormal
operator and an operator of class A were introduced by V.
Istr\u{a}\c{t}escu in \cite{Ist67} and by T. Furuta, M. Ito
and T. Yamazaki in \cite{FIY98}, respectively. We refer the
reader to Theorem~\ref{yama} for an explanation of why
operators of class A appear naturally.

\begin{figure}
\centering
\begin{tikzpicture}
\matrix (m) [matrix of math nodes,row
sep=2em,column sep=1em,minimum
width=2em]{\textit{normal}& \textit{quasinormal}
& \textit{subnormal} & \textit{hyponormal}
   \\}; \path[-stealth,auto]
(m-1-1) edge[subseteq] node {}%{$q$}
(m-1-2)
(m-1-2) edge[subseteq] node {}%{$q$}
(m-1-3)
(m-1-3) edge[subseteq] node {}%{$q$}
(m-1-4)
(m-1-2) edge[subseteq] node {}%{$q$}
(m-1-3);
\end{tikzpicture}
\begin{tikzpicture}
\matrix (m) [matrix of math nodes,row
sep=0.5em,column sep=1.7em,minimum width=2em]{ &
\textit{$p$-hyponormal}& &
   \\
\textit{invertible p-hyponormal} & &
\textit{class A} & \textit{paranormal}
   \\
  & \textit{log-hyponormal} & &
   \\}; \path[-stealth,auto]
(m-2-1) edge[subseteq] node {}%{$q$}
(m-1-2)
(m-2-1) edge[subseteq] node {}%{$q$}
(m-3-2)
(m-1-2) edge[subseteq] node {}%{$q$}
(m-2-3)
(m-3-2) edge[subseteq] node {}%{$q$}
(m-2-3)
(m-2-3) edge[subseteq] node {}%{$q$}
(m-2-4);
\end{tikzpicture}
   \vspace{-2ex} \caption{Inclusion relations between the
classes of operators under consideration.} \label{inkluzje}
\end{figure}
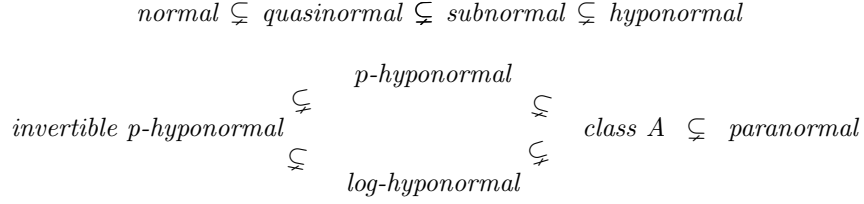

The present work is a continuation of the
article \cite{P-S21}, in which the authors
solved affirmatively the problem posed by R. E.
Curto, S. H. Lee and J. Yoon (see
\cite[Problem~1.1]{Curto20}). In fact, the
following more general result has been proven.
   \begin{theorem}[{\cite[Theorem~1.2]{P-S21}}]\label{twPS}
If $T\in \ogr(\hh)$ is a subnormal operator such
that $T^n$ is quasinormal, where $n$ is a
positive integer, then $T$ is quasinormal.
   \end{theorem}
In \cite[Section~3]{P-S21} we gave two proofs of
this theorem using quite different techniques.
The first technique appeals to the theory of
operator monotone functions with emphases on
Hansen's inequality. The second relies on the
theory of (scalar and operator) moment problems.
The origins of the second technique go back to
the celebrated Embry's characterization of
subnormal operators expressed in terms of the
Stieltjes operator moment problem (see
\cite{Embry73}; see also \cite{Lam76,Ag85}).

Problem~1.1 in \cite{Curto20} can be naturally
generalized in two directions: by enlarging the
class of $n$th roots, and by allowing the
operators in question to be closed and
unbounded.
   \begin{prob}[{see \cite[Problem 5.1]{P-S21}}] \label{prob2}
Let $T$ be a subnormal $($hyponormal, etc.$)$
operator which is bounded, or unbounded and
closed. Assume that for some integer $n \Ge 2$,
$T^n$ is quasinormal. Does it follow that $T$ is
quasinormal?
   \end{prob}
It turns out that the first technique of proving
Theorem~\ref{twPS} which appeals to operator
monotone functions is more suitable for bounded
operators. Namely, we will prove the following
theorem, which is the first of the three main
results of this paper.
   \begin{theorem}\label{twhyp}
Let $T\in \ogr(\hh)$ be of class A $($in
particular, $p$-hyponormal or
log-hyponormal\/$)$ and $n$ be an integer
greater than $1$ such that $T^n$ is quasinormal.
Then $T$ is quasinormal.
   \end{theorem}
Theorem~\ref{twhyp} follows from the second
statement of Theorem~\ref{rowm} in
Section~\ref{Sec.4}. The first statement of
Theorem~\ref{rowm} relates Embry's description
of quasinormal operators (see
Theorem~\ref{embry}) to a certain chain of
inequalities characterizing operators of class A
(see Theorem~\ref{yama}).

Results similar to those in Theorems~\ref{twPS} and
\ref{twhyp} for $n$th roots of normal operators have been
known for a long time. Namely, J. G. Stampfli proved that a
hyponormal $n$th root of a normal operator is normal (see
\cite[Theorem~5]{sta62}). T.~Ando improved this result
showing that a paranormal $n$th root of a normal operator is
normal (see \cite[Theorem~6]{Ando72}). However, a hyponormal
$n$th root of a subnormal operator need not be subnormal (see
\cite[pp.\ 378/379]{sta66}). It turns out that normal
operators and non-normal quasinormal operators can have
non-normal and non-quasinormal $n$th roots, respectively. A
more detailed discussion on this topic can be found in
Section~\ref{Sec.6}. Other questions concerning square roots
(or more generally $n$th roots) in selected classes of
operators have been studied at least since the early 1950's
(see e.g.,\
\cite{hal53,hal54,put57,Emb68,Gil74,Keo84,wog85,con87,Dug93,Ki-Ko22,MPR}).

To prove Theorem~\ref{twhyp}, we will need the
following theorem, which is the second of the
three main results of this paper. It generalizes
\cite[Lemma~3.7]{pp18} in two directions. First,
it removes the injectivity assumption, and
second, it replaces commutativity by a more
general intertwining relation. We give the proof
of Theorem~\ref{bblem} in Section~\ref{Sec.3}.
This theorem is no longer true if the operators
$A$ and $B$ do not satisfy the condition
$A^*A\Le B$, even if $\kk=\hh$ and $C=B$ (see
\cite[Example~3.10]{pp18}).
   \begin{theorem} \label{bblem}
Let $\hh$ and $\kk$ be complex Hilbert spaces,
$A\in \ogr(\hh,\kk)$, $B\in \ogr_+(\hh)$ and
$C\in \ogr_+(\kk)$. Suppose that $\alpha, \beta$
are distinct positive real numbers. Then the
following conditions are equivalent{\em :}
   \begin{itemize}
   \item[(i)]
$A^*A \Le B$ and $A^*C^sA =B^{s+1}$ for
$s=\alpha,\beta$,
   \item[(ii)] $A^*A= B$ and $AB=CA$.
   \end{itemize}
   \end{theorem}
The above result covers the case of $n$-tuples
of noncommuting operators. As shown below,
Theorem~\ref{bblem} implies Theorem~\ref{sewer}.
Since the converse implication is obvious, both
theorems are logically equivalent.
   \begin{theorem} \label{sewer}
Fix a positive integer $n$. Let $\hh, \kk_i$ be
complex Hilbert spaces, $A_i \in \ogr(\hh,
\kk_i)$, $B\in \ogr_+(\hh)$ and $C_i \in
\ogr_+(\kk_i)$, where $i=1, \ldots, n$. Suppose
that $\alpha, \beta$ are distinct positive real
numbers. Then the following conditions are
equivalent{\em :}
   \begin{enumerate}
   \item[(i)]
$A_1^*A_1 + \cdots + A_n^*A_n \Le B$ and
$A_1^*C_1^sA_1 + \cdots + A_n^*C_n^sA_n =
B^{s+1}$ for $s=\alpha,\beta$,
   \item[(ii)]
$A_1^*A_1 + \cdots + A_n^*A_n = B$ and $A_i B=
C_iA_i$ for $i=1, \ldots,n$.
   \end{enumerate}
   \end{theorem}
   \begin{proof}
Apply Theorem~\ref{bblem} to the quadruple
$(\kk,A,B,C)$ defined by $\kk:=\kk_1 \oplus
\cdots \oplus \kk_n$, $C:=C_1 \oplus \cdots
\oplus C_n$ and $Ah:= A_1h \oplus \cdots \oplus
A_n h$ for $h \in \hh$.
   \end{proof}
It is worth pointing out that
Theorem~\ref{bblem} allows us to obtain a useful
criterion for the quasinormality of arbitrary
operators (without assuming injectivity).
   \begin{theorem} \label{achtwdw}
Let $A\in \ogr(\hh)$, $B\in \ogr_+(\hh)$ and $\alpha, \beta$
be distinct positive real numbers. Then the following
conditions are equivalent{\em :}
   \begin{itemize}
   \item[(i)]
$A^*A \Le B$ and $A^*B^sA =B^{s+1}$ for
$s=\alpha,\beta$,
   \item[(ii)] $A$ is quasinormal and $B=|A|^2$.
   \end{itemize}
   \end{theorem}
   \begin{proof}
It follows from Theorem~\ref{bblem} with $\kk=\hh$ and
$C=B$ that the condition (i) is equivalent to the
conjunction of two equalities $B=|A|^2$ and
$A(A^*A)=(A^*A)A$.
   \end{proof}
As shown in Section~\ref{Sec.5}, the second
technique used in the proof of
Theorem~\ref{twPS} which is based on the theory
of moments is better suited to unbounded (i.e.,
not necessarily bounded) subnormal operators.
First, we need to define the unbounded
counterparts of the concepts of quasinormality
and subnormality. Given a linear operator $T$ in
$\hh$, we denote by $\ddd(T)$, $\jd{T}$,
$\ran(T)$ and $T^*$ the domain, the kernel, the
range and the adjoint of $T$, respectively.
Following \cite{Kauf83} (cf.\
\cite{Sto-Szaf89}), we say that a closed densely
defined operator $T$ in $\hh$ is {\em
quasinormal} if $T(T^*T)=(T^*T)T$, or
equivalently (see \cite[Theorem~3.1]{jabl14}) if
and only if $E(\varDelta) T \subseteq
TE(\varDelta)$ for all Borel subsets $\varDelta$
of the nonnegative part of the real line, where
$E$ is the spectral measure of $|T|$. A densely
defined operator $T$ in $\hh$ is said to be {\em
subnormal} if there exists a complex Hilbert
space $\kk$ and a normal operator $N$ in $\kk$
such that $\hh \subseteq \kk$ (isometric
embedding), $\ddd(T)\subseteq \ddd(N)$ and $Th =
Nh$ for all $h \in \ddd(T)$. Such $N$ is called
a {\em normal extension} of $T$. The foundations
of the theory of unbounded subnormal operators
were developed in
\cite{Sto-Szaf85,Sto-Szaf89,Sto-Szaf89III,Sto-Szaf98}.

We are now ready to state the last of the three
main results of this paper. Its proof is given
in Section~\ref{Sec.5}.
   \begin{theorem}\label{twsub}
Let $T$ be a closed densely defined operator in
$\hh$ and $n$ be an integer greater than $1$.
Suppose that $T$ is subnormal and $T^n$ is
quasinormal. Then $T$ is quasinormal.
   \end{theorem}
   \section{\label{Sec.2}Preliminaries}
In this paper, we use the following notation.
The fields of real and complex numbers are
denoted by $\mathbb{R}$ and $\mathbb{C}$,
respectively. The symbols $\zbb_+$, $\natu$ and
$\rbb_+$ stand for the sets of nonnegative
integers, positive integers and nonnegative real
numbers, respectively. Given a set $\varDelta
\subseteq \cbb$, we write $\varDelta^*=\{\bar z
\colon z \in \varDelta\}$. Denote by
$\borel(\varOmega)$ the $\sigma$-algebra of all
Borel subsets of a topological Hausdorff space
$\varOmega$.

A sequence $\{\gamma_n\}_{n=0}^{\infty}$ of real
numbers is said to be a {\em Stieltjes moment
sequence} if there exists a positive Borel
measure $\mu$ on $\rbb_+$ such that
   \begin{align} \label{hamb}
\gamma_n = \int_{\rbb_+} t^n d \mu(t), \quad
n\in \zbb_+.
   \end{align}
A positive Borel measure $\mu$ on $\rbb_+$ satisfying
\eqref{hamb} is called a {\em representing measure} of
$\{\gamma_n\}_{n=0}^{\infty}$. If
$\{\gamma_n\}_{n=0}^{\infty}$ is a Stieltjes moment sequence
which has a unique representing measure, then we say that
$\{\gamma_n\}_{n=0}^{\infty}$ is {\em determinate}. It is
well known that if a Stieltjes moment sequence has a
representing measure with compact support, then it is
determinate. The reader is referred to \cite{B-C-R} for the
foundations of the theory of moment problems.

Let $\ascr$ be a $\sigma$-algebra of subsets of
a set $\varOmega$ and let $F\colon \ascr \to
\ogr(\hh)$ be a {\em positive operator valued
measure} (a {\em POV measure} for brevity), that
is $\langle F (\cdot)f, f\rangle$ is a positive
measure for every $f \in \hh$. Denote by
$L^1(F)$ the vector space of all
$\ascr$-measurable functions $f\colon \varOmega
\to \cbb$ such that $\int_{\varOmega} |f(x)|
\langle F(\D x)h, h\rangle < \infty$ for all
$h\in \hh$. Then for every $f\in L^1(F)$, there
exists a unique operator $\int_\varOmega f \D F
\in \ogr(\hh)$ such that (see e.g.,\
\cite[Appendix]{Sto92})
   \begin{align*}
\Big\langle\int_\varOmega f \D F h, h\Big\rangle
= \int_\varOmega f(x) \langle F(\D x)h,
h\rangle, \quad h\in\hh.
   \end{align*}
If a POV measure $F$ is normalized, that is $F (\varOmega) =
I$, then $F$ is called a {\em semispectral measure}. Observe
that if $F$ is a {\em spectral measure}, that is $F$ is a
semispectral measure such that $F(\varDelta)$ is an
orthogonal projection for every $\varDelta \in \ascr$, then
$\int_\varOmega f \D F$ coincides with the usual spectral
integral. If $F$ is the spectral measure of a normal operator
$T$, then we write $f(T)=\int_{\cbb} f \D F$ for a Borel
function $f\colon \cbb \to \cbb$; the map $f \mapsto f(T)$ is
called the Stone-von Neumann functional calculus. We refer
the reader to \cite{Rud73,Weid80,Bir-Sol87,Sch12} for the
necessary information on spectral integrals, including the
spectral theorem for normal operators and the Stone-von
Neumann functional calculus, which we will need in this
paper.

In the proofs of Theorems~\ref{twhyp} and \ref{twsub}, we use
the following characterizations of quasinormal operators (the
``moreover'' part of Theorem~\ref{embry} follows from the
observation that by \eqref{kul-3}, $E$ is the spectral
measure of $T^*T$)
   \begin{theorem}[\cite{Embry73,jabl14}] \label{embry}
Let $T$ be a closed densely defined operator in
$\hh$. Then the following conditions are
equivalent{\em :}
   \begin{itemize}
   \item[(i)] $T$ is quasinormal,
   \end{itemize}
   \begin{itemize}
   \item[(ii)] $T^{*k}T^{k}=(T^*T)^k$ for $k\in \zbb_+$,
   \item[(iii)]
$(T^{*k}T^{k})^{\frac1k}= T^*T$ for $k\in
\natu$,
   \item[(iv)] there exists a spectral measure $E\colon
\borel(\rbb_+) \to \ogr(\hh)$ such that
   \begin{equation} \label{kul-3}
T^{*k}T^k = \int_{\mathbb{R}_{+}} x^k E(\D x),
\quad k \in \zbb_+.
   \end{equation}
   \end{itemize}
Moreover, the spectral measure $E$ in {\em (iv)} is unique
and if $T\in \ogr(\hh)$, then
   \begin{align*}
E((\|T\|^2, \infty))=0.
   \end{align*}
   \end{theorem}
The above characterizations of quasinormal operators were
invented by M. R. Embry for bounded operators (see {\cite[p.\
63]{Embry73}}) and then extended to unbounded ones by Z. J.
Jab{\l}o{\'n}ski, I. B. Jung and the second-named author (see
\cite[Theorem~3.6]{jabl14}; cf.\ \cite{uch93}). Although the
condition (ii) looks more elaborate than $T(T^*T)=(T^*T)T$,
it allows us to use the techniques related to positive
operators including spectral theorem, the Stone-von Neumann
functional calculus, operator monotone and operator convex
functions and operator inequalities.

The condition (ii) of Theorem~\ref{embry} leads to the
problem of reduced Embry's characterization of quasinormality
(see \cite[Problem~1.4]{P-S21}). This problem, to some extent
related to the theory of operator monotone and operator
convex functions, has been studied by several authors (see
e.g.,\ \cite{uch93,uch01,Jib10,jabl14,pp16,pp18}). In
particular, it was shown in \cite[Example~5.5]{jabl14} (see
also \cite[Theorem~4.3]{pp16}) that for every integer $n\Ge
2$, there exists an operator $T\in \ogr(\hh)$ such that
  \begin{equation} \label{qqq}
\text{$T^{*n}T^{n}=(T^{*} T)^n$ and $T^{*k}T^{k}
\neq (T^{*} T)^k$ for all
$k\in\{2,3,4,\ldots\}\setminus \{n\}$.}
  \end{equation}

The following result, which is closely related to
Theorem~\ref{embry}(iii), plays a key role in the proof of
Theorem~\ref{twhyp} (see Theorem~\ref{rowm}). In particular,
it shows that an operator $T\in \ogr(\hh)$ is of class A if
and only if the sequence
$\{(T^{*k}T^k)^{\frac1k}\}_{k=1}^{\infty}$ is monotonically
increasing.
   \begin{theorem}[{\cite[Theorem~1]{Ito02}; cf.\ \cite[Theorems~2 \&
3]{Ito99} and \cite[Theorem~1]{Yama99}}]
\label{yama} If $T\in \ogr(\hh)$ is of class A
$($in particular, $p$-hyponormal or
log-hyponormal\/$)$, then the sequence
$\{(T^{*k}T^k)^{\frac1k}\}_{k=1}^{\infty}$
$($resp.,
$\{(T^{k}T^{*k})^{\frac1k}\}_{k=1}^{\infty}$$)$
is monotonically increasing $($resp.,
monotonically decreasing$)$, that is
   \begin{equation*}
T^*T \Le (T^{*2}T^2)^{\frac12} \Le
(T^{*3}T^3)^{\frac13} \Le \ldots,
   \end{equation*}
and
   \begin{equation*}
TT^* \Ge (T^2T^{*2})^{\frac12} \Ge
(T^3T^{*3})^{\frac13} \Ge \ldots.
   \end{equation*}
   \end{theorem}
We conclude this section with a more
detailed discussion of
Figure~\ref{inkluzje}. That hyponormal
operators are of class A, can be
justified as follows. If $T^*T \Ge TT^*$,
then $T^*(T^*T)T \Ge T^*(TT^*)T$ and thus
by the L\"owner-Heinz inequality with
exponent $\frac 1 2$ (see
\cite{Lo34,He51}), $(T^{*2}T^2)^{\frac12}
\Ge T^*T$. This fact also follows from a
more general result due to T. Yamazaki,
which shows in particular that
$p$-hyponormal operators with $p\in
(0,1]$ are of class A (see
\cite[Theorem~1(i)]{Yama99}). In fact,
$p$-hyponormal operators are always of
class A because $p$-hyponormal operators
are $q$-hyponormal whenever $0 < q < p <
\infty$ (apply the L\"owner-Heinz
inequality with exponent $\frac q p$). It
is well known that invertible
$p$-hyponormal operators are
log-hyponormal (see \cite[Theorem~1 in \S
3.4.2]{Fur01}). However, one can
construct a log-hyponormal operator that
is not $p$-hyponormal for any $p\in
(0,\infty)$ (see
\cite[Example~12]{Tana99}). In turn,
every log-hyponormal operator is of class
A and every class A operator is
paranormal (see \cite[Theorem~1 in \S
3.5.1]{Fur01}). Note also that strict
inclusions appear in
Figure~\ref{inkluzje} only if $\hh$ is
infinite dimensional (see
\cite[Theorem~2.2]{Ist67}). More
information on the classes of bounded
operators considered in this paper can be
found in \cite{con91,Fur01}.
   \section{\label{Sec.3}Proof of the intertwining theorem}
In this section we give a proof of Theorem~\ref{bblem} based
on a recent result of the authors (see
\cite[Theorem~4.2]{P-S21}). In fact, we need a version of it
for positive operator valued measures that are not
necessarily normalized.
   \begin{theorem} \label{pupra}
Let $T\in \ogr(\hh)$ be a positive injective
operator and $\alpha,\beta$ be distinct positive
real numbers. Assume that $F\colon
\borel(\rbb_+) \to \ogr(\hh)$ is a POV measure
with compact support. Then the following
conditions are~equivalent{\em :}
   \begin{enumerate}
   \item[(i)] $F$ is the spectral
measure of $T$,
   \item[(ii)] $T^p =\int_{\rbb_+}
x^p F(\D x)$ for $p= \alpha, \beta$ and
$F(\rbb_+)\Le I$.
   \end{enumerate}
   \end{theorem}
   \begin{proof}
   (i)$\Rightarrow$(ii) It is obvious.

   (ii)$\Rightarrow$(i) Let $E\colon
\borel(\rbb_+) \to \ogr(\hh) $ be the spectral
measure of $T$. Since $F(\rbb_+)\Le I$, the map
$\widetilde{F}\colon \borel(\rbb_+) \to
\ogr(\hh)$ defined by
   \begin{align*}
\widetilde F(\varDelta) = F(\varDelta) +
\delta_0(\varDelta)(I-F(\rbb_+)),\quad \varDelta
\in \borel(\rbb_+),
   \end{align*}
is a semispectral measure. It is easily seen
that $\widetilde{F}$ has compact support and
   \begin{align*}
T^p =\int_{\rbb_+} x^p \widetilde F(\D x), \quad
p= \alpha, \beta.
   \end{align*}
By \cite[Theorem~4.2]{P-S21} and
\cite[Theorem]{P-S22}, $\widetilde F$ is the
spectral measure of $T$, which yields
   \begin{align}\label{ewff}
E = \widetilde F = F + \delta_0(I-F(\rbb_+)).
   \end{align}
Since $\nul(T) = \{0\}$, we see that
$E(\{0\})=0$ and thus
   \begin{align*}
0 = E(\{0\}) \overset{\eqref{ewff}}{=} F(\{0\})
+ (I-F(\rbb_+)).
   \end{align*}
This implies that $F(\rbb_+)=I$ and consequently
$\widetilde F = F$. Therefore $F$ is the
spectral measure of $T$. This completes the
proof.
   \end{proof}
   We also need the following result which gives a necessary
and sufficient condition for equality to hold in a
Kadison-type inequality (cf.\ \cite[Lemma~3.1]{P-S22-b}).
   \begin{lemma} \label{kadlemma}
Let $\hh$ and $\kk$ be complex Hilbert spaces, $V \in
\ogr(\hh,\kk)$ and $T\in \ogr(\kk)$. Suppose that $\|V\| \Le
1$. Then the following inequality is valid{\em :}
   \begin{equation}\label{kadP}
(V^*TV)^*(V^*TV) \Le V^*T^*TV.
   \end{equation}
Moreover, equality holds in \eqref{kadP} if and only if
$TV=VV^*TV$.
   \end{lemma}
   \begin{proof}
Since $\|V^*\|\Le 1$, we deduce that $I_{\kk}-VV^* \Ge 0$,
and therefore
   \begin{align} \label{kadP2}
V^*T^*TV-(V^*T^*V)(V^*TV)& = (TV)^* (I_{\kk}-VV^*)TV\Ge 0.
   \end{align}
This yields \eqref{kadP}.

It remains to prove the ``moreover'' part. It
follows from \eqref{kadP2} that equality holds
in \eqref{kadP} if and only if
   \begin{align*}
\ran(TV)\subseteq\jd{(I_{\kk}-VV^*)^\frac{1}{2}}=\jd{I_{\kk}-VV^*},
   \end{align*}
or equivalently if and only if $TV=VV^*TV$.
   \end{proof}
   \begin{proof}[Proof of Theorem~\ref{bblem}]
(i)$\Rightarrow$(ii) It follows from the
inequality $A^*A \Le B$ and the Douglas
factorization theorem (see
\cite[Theorem~1]{Dou66}) that there exists an
operator $Q\in \ogr(\hh,\kk)$ such that
   \begin{equation} \label{doulem}
\|Q\|\Le 1\quad \text{and} \quad
A=QB^\frac{1}{2}.
   \end{equation}
Since $A^*C^sA =B^{s+1}$ for $s=\alpha,\beta$,
we infer from \eqref{doulem} that
   \begin{equation}\label{vvvp}
B^\frac{1}{2}Q^*C^sQB^\frac{1}{2}=B^\frac{1}{2}B^{s}B^\frac{1}{2},
\quad s=\alpha,\beta.
   \end{equation}
Set $\hh_0 = \overline{\ran(B)}$. Define the
operator $Q_0\in \ogr(\hh_0,\kk)$ by $Q_0 h =Qh$
for $h\in \hh_0$. Observe that $Q_0^*\in
\ogr(\kk,\hh_0)$ is given by
   \begin{align}  \label{takijaki}
Q_0^*=P_{\hh_0}Q^*,
   \end{align}
where $P_{\hh_0} \in \ogr(\hh)$ is the
orthogonal projection of $\hh$ onto $\hh_0$.
Note that $\hh_0$ reduces $B$ to $B|_{\hh_0}\in
\ogr_+(\hh_0)$ and that the identity
\eqref{vvvp} is equivalent to
   \begin{align}  \label{prudtr}
\langle Q^*C^sQB^\frac{1}{2}h,
B^\frac{1}{2}h^\prime \rangle=\langle
B^{s}B^\frac{1}{2}h,
B^\frac{1}{2}h^\prime\rangle, \quad
h,h^\prime\in \hh, \; s=\alpha,\beta.
   \end{align}
Since $\overline{\ran(B)} =
\overline{\ran(B^\frac{1}{2})}$, \eqref{prudtr}
holds if and only if
   \begin{align*}
\langle Q^*C^sQh_0,h_0^\prime\rangle =\langle
B^{s}h_0,h_0^\prime\rangle, \quad
h_0,h_0^\prime\in \hh_0,\; s=\alpha,\beta.
   \end{align*}
Combined with \eqref{takijaki}, this yields
   \begin{align}\label{ccp}
Q_0^*C^sQ_0= (B|_{\hh_0})^{s}, \quad
s=\alpha,\beta.
   \end{align}
Let $E\colon \borel(\rbb_+) \to \ogr(\kk) $ be
the spectral measure of $C$. Then \eqref{ccp}
implies that
   \begin{align*}
(B|_{\hh_0})^s = \int_{\rbb_+} x^s F(\D x),
\quad s= \alpha, \beta,
   \end{align*}
where $F\colon \borel(\rbb_+) \to \ogr(\hh_0) $
is the POV measure with compact support defined
by
   \begin{align} \label{eqeq}
F(\varDelta) = Q_0^*E(\varDelta)Q_0,\quad
\varDelta \in \borel(\rbb_+).
   \end{align}
It follows from \eqref{doulem} that $\|Q_0\| \Le
1$. Since $\jd{B|_{\hh_0}}=\{0\}$ and
   \begin{align*}
F(\rbb_+)=Q_0^*E(\rbb_+)Q_0 = Q_0^*Q_0\Le I_{\hh_0},
   \end{align*}
we deduce from Theorem~\ref{pupra} that $F$ is
the spectral measure of $B|_{\hh_0}$. In
particular, we have
   \begin{align}\label{qoi}
I_{\hh_0} = F(\rbb_+) \overset{\eqref{eqeq}}=
Q_0^*E(\rbb_+)Q_0 = Q_0^*Q_0.
   \end{align}
Note now that
   \begin{align*}
Q_0^*E(\varDelta)^2Q_0 \overset{\eqref{eqeq}}{=}
F(\varDelta)=(F(\varDelta))^2
\overset{\eqref{eqeq}}{=}
(Q_0^*E(\varDelta)Q_0)^2,\quad \varDelta \in
\borel(\rbb_+).
   \end{align*}
By Lemma~\ref{kadlemma}, this gives
   \begin{align*}
E(\varDelta)Q_0 {=} Q_0Q_0^*E(\varDelta) Q_0
\overset{\eqref{eqeq}}{=} Q_0 F(\varDelta),\quad
\varDelta \in \borel(\rbb_+).
   \end{align*}
Using \cite[Proposition~5.15]{Sch12}, we obtain
   \begin{align} \label{namur}
CQ_0 {=} Q_0B|_{\hh_0}.
   \end{align}
Hence, we have
   \begin{align*}
ABh_0 & \overset{\eqref{doulem}}{=}Q
B^\frac{1}{2} Bh_0
   \\
&\hspace{1ex} =Q_0
B|_{\hh_0}(B|_{\hh_0})^\frac{1}{2}h_0
   \\
&\hspace{-.5ex}\overset{\eqref{namur}}= CQ_0
(B|_{\hh_0})^\frac{1}{2} h_0
   \\
& \overset{\eqref{doulem}}{=}CAh_0,\quad
h_0\in\hh_0.
   \end{align*}
This shows that
   \begin{align}\label{naho}
AB|_{\hh_0}=CA|_{\hh_0}.
   \end{align}
However, $\hh_0^\perp=\jd{B}$ and thus
$A|_{\hh_0^\perp}=0$ because
   \begin{equation*}
\|Ah\|^2 = \langle A^{*}Ah, h \rangle \Le
\langle Bh,h \rangle =0,\quad h \in \jd{B}.
   \end{equation*}
As a consequence, we get
     \begin{align}\label{nahop}
    AB|_{\hh_0^\perp}=0=CA|_{\hh_0^\perp}.
\end{align}
It follows from \eqref{naho} and \eqref{nahop}
that $AB=CA$.

It remains to show that $A^*A = B$. For, note
that $\hh_0$ reduces $A^*A$ and $B$, and
   \begin{align*}
\langle A^*Ah_0, h_0^\prime \rangle &= \langle
Ah_0, Ah_0^\prime \rangle
   \\
&
\hspace{-1ex}\overset{\eqref{doulem}}{=}\langle
Q B^\frac{1}{2} h_0, Q B^\frac{1}{2}
h_0^\prime\rangle
   \\
&=\langle Q_0 B^\frac{1}{2}
h_0, Q_0 B^\frac{1}{2} h_0^\prime \rangle
   \\
& = \langle Q_0^*Q_0 B^\frac{1}{2} h_0,
B^\frac{1}{2} h_0^\prime \rangle
   \\
& \hspace{-1.5ex}\overset{\eqref{qoi}}{=}
\langle B^\frac{1}{2} h_0, B^\frac{1}{2}
h_0^\prime \rangle
   \\
&= \langle Bh_0, h_0^\prime \rangle,\quad h_0,
h_0^\prime\in \hh_0,
   \end{align*}
which implies that $A^*A|_{\hh_0}=B|_{\hh_0}$.
Clearly, $A^*A|_{\hh_0^\perp} = 0 =
B|_{\hh_0^\perp}$, so $A^*A = B$.

   (ii)$\Rightarrow$(i) It suffices to use the
fact that $AB=CA$ implies $AB^s=C^sA$ for all
positive real number $s$. This completes the
proof.
  \end{proof}
   \begin{rem}
As shown in Section~\ref{Sec.1}, Theorem~\ref{bblem} implies
Theorem~\ref{achtwdw}. However, the authors see no direct way
to deduce Theorem~\ref{bblem} from Theorem~\ref{achtwdw} (the
famous Berberian matrix trick does not give the expected
result). On the other hand, from \cite[Lemma~3.7]{pp18} one
can deduce its version in which the operator $B$ is not
assumed to be injective. Indeed, suppose that the condition
(i) of Theorem~\ref{achtwdw} hold. We show that $A^*A=B$, $A$
commutes with $B$, $\jd{B}$ reduces $A$ and $A|_{\jd{B}}=0$
(the converse implication is obvious). First, we claim that
$A|_{\nul(B)}=0$. Indeed, if $h\in \nul(B)$, then
   \begin{equation*}
\|Ah\|^2 = \langle A^{*}Ah, h \rangle \Le
\langle Bh,h \rangle =0,
   \end{equation*}
so $h\in \nul(A)$. Thus the operators $A$ and
$B$ have the block matrix representations
   \begin{equation}\label{deks}
A = \left[\begin{matrix} \tilde A & 0 \\ C & 0
\end{matrix} \right] \quad  \text{and} \quad B =
\left[ \begin{matrix} \tilde B & 0 \\ 0 & 0
\end{matrix} \right]
   \end{equation}
with respect to the orthogonal decomposition
$\hh = \overline{\ran(B)} \oplus \nul (B)$,
where
   \begin{align*} \label{duwnr}
\tilde A = P A|_{\overline{\ran (B)}}, \quad
\tilde B = B|_{\overline{\ran (B)}}, \quad
C=(I-P)A|_{\overline{\ran( B)}},
   \end{align*}
and $P$ is the orthogonal projection of $\hh$
onto $\overline{\ran (B)}$. This implies that
   \begin{equation} \label{deksm}
\left[\begin{matrix} \tilde A^* \tilde A+C^*C &
0 \\0 & 0 \end{matrix} \right] = \left[
\begin{matrix} \tilde A^* & C^*
\\0 & 0 \end{matrix} \right] \left[\begin{matrix}
\tilde A & 0 \\C & 0 \end{matrix} \right] = A^*A
\Le B = \left[\begin{matrix} \tilde B & 0
\\0 & 0 \end{matrix} \right].
   \end{equation}
Hence
   \begin{align} \label{nurw}
\tilde A^* \tilde A+C^*C\Le \tilde B,
   \end{align}
which yields $\tilde A^* \tilde A\Le \tilde B$.
Observe that
   \begin{equation*}\label{deksmu}
A^*B^s A = \left[\begin{matrix} \tilde A^* & C^*
\\0 & 0 \end{matrix} \right] \left[\begin{matrix}
\tilde B^s & 0\\0 & 0 \end{matrix} \right]
\left[\begin{matrix} \tilde A & 0 \\C & 0
\end{matrix} \right] = \left[\begin{matrix}
\tilde A^*\tilde B^s\tilde A & 0 \\0 & 0
\end{matrix} \right], \quad s \in (0,\infty),
   \end{equation*}
and
   \begin{equation*}\label{deksmb}
B^{s+1}=\left[ \begin{array}{cc} \tilde B^{s+1}
& 0 \\0 & 0 \end{array} \right], \quad s \in
(0,\infty).
   \end{equation*}
Combined with the equality in Theorem~\ref{achtwdw}(i), this
shows that $\tilde A^*\tilde B^s \tilde A = \tilde B^{s+1}$
for $s=\alpha,\beta$. Clearly $\nul(\tilde B) = \{0\}$.
Therefore, by \cite[Lemma~3.7]{pp18}, $\tilde A$ commutes
with $\tilde B$ and $\tilde A^*\tilde A= \tilde B$. This and
\eqref{nurw} implies that $C=0$. Thus by \eqref{deks},
$\nul(B)$ reduces $A$. Finally, it follows from \eqref{deks}
and \eqref{deksm} that $AB=BA$ and $A^* A= B$, which
completes the proof.
   \hfill $\diamondsuit$
   \end{rem}
   \section{\label{Sec.4}Class A $n$th roots of bounded quasinormal operators}
The main purpose of this section is to prove
Theorem~\ref{twhyp}. In view of Embry's characterization of
quasinormal operators (see Theorem~\ref{embry}(iii)),
Problem~\ref{prob2} for operators $T$ of class A is closely
related to the question when the monotonically increasing
sequence $\big\{(T^{*k}T^k)^{\frac1k}\big\}_{k=1}^{\infty}$
appearing in Theorem~\ref{yama} is constant. The answer given
in Theorem~\ref{rowm} below shows that this is the case when
the distance between equal terms of the sequence is at least
two (see also Problem~\ref{as}). As a consequence, we obtain
an affirmative solution to Problem~\ref{prob2} for operators
of class A (see the second statement of Theorem~\ref{rowm}).
   \begin{theorem} \label{rowm}
If $T\in \ogr(\hh)$ is of class A $($in
particular, $p$-hyponormal or
log-hyponormal\/$)$, then any of the following
statements implies that $T$ is quasinormal{\em
:}
   \begin{enumerate}
   \item[(i)] $(T^{*n}T^n)^{\frac1n} = (T^{*k}T^k)^{\frac1k}$
for some positive integers $k,n$ such that
$k-n\Ge 2$,
   \item[(ii)] $T^n$ is quasinormal for some positive integer $n$.
   \end{enumerate}
   \end{theorem}
   \begin{proof}
Suppose that (i) holds. In view of Theorem~\ref{yama}, there
is no loss of generality in assuming that $k=n+2$. It also
follows from Theorem~\ref{yama} that
   \begin{align}\label{ineq}
T^{*}T \Le \ldots \Le (T^{*n}T^n)^{\frac1n}
   \end{align}
and
   \begin{align} \label{rysa}
(T^{*j}T^{j})^{\frac1j}
&=(T^{*n}T^n)^{\frac1n},\quad j=n+1, n+2.
   \end{align}
Set $D=(T^{*n}T^n)^{\frac1n}$. By \eqref{ineq},
we see that $T^*T\Le D$. Note further that
   \begin{equation}  \label{latfi}
T^*(T^{*n}T^{n}) T = T^{*(n+1)}T^{n+1}
\overset{\eqref{rysa}} =
(T^{*n}T^{n})^{\frac{n+1}{n}}
   \end{equation}
and
   \begin{equation*}
T^*(T^{*n}T^{n})^{\frac{n+1}{n}} T
\overset{\eqref{latfi}}= T^{*(n+2)} T^{n+2}
\overset{\eqref{rysa}}=(T^{*n}T^{n})^\frac{n+2}{n}.
   \end{equation*}
Therefore, we have
   \begin{align*}
T^*D^s T = D^{s+1}, \quad s=n, n+1.
   \end{align*}
Applying Theorem~\ref{achtwdw} to $(A,B)=(T,D)$,
we conclude that $T$ is quasinormal.

Assume now that (ii) holds. Applying
Theorem~\ref{embry}(iii) to $T^n$, we deduce
that
   \begin{align*}
(T^{*nl}T^{nl})^{\frac1{nl}}=(T^{*n}
T^n)^{\frac1n}, \quad l\in \natu,
   \end{align*}
which implies that $T$ satisfies (i). This
completes the proof.
   \end{proof}
The statement (i) of Theorem~\ref{rowm} suggests
the following problem which is of some
independent interest.
   \begin{prob}[Flatness problem] \label{as}
Let $T\in \ogr(\hh)$ be a class A operator.
Assume that for some integer $n\Ge 2$,
$(T^{*n}T^n)^{\frac1n}=
(T^{*(n+1)}T^{n+1})^{\frac1{n+1}}$. Does it
follow that the sequence
$\{(T^{*j}T^j)^{\frac1j}\}_{j=1}^\infty$ is
constant $($equivalently, $T$ is
quasinormal\/$)${\em ?}
   \end{prob}
Note that Problem~\ref{as} is interesting only for integers
$n \Ge 2$ because for $n=1$ the answer is negative (see
\eqref{qqq}).

It is worth noting that by Theorem~\ref{embry},
any quasinormal operator $X\in \ogr(\hh)$
satisfies the single equation
   \begin{equation}\label{xukl-5}
X^{*\kappa}X^{\kappa}=(X^*X)^{\kappa},
   \end{equation}
where $\kappa$ is a fixed integer greater than $1$,
but not conversely (see \eqref{qqq}). It turns out
that the class of operators satisfying \eqref{xukl-5}
for a single $\kappa$ can successfully replace
quasinormal operators in the predecessor of the
implication in Theorem \ref{twhyp} (cf.\
\cite[Theorem~4.1]{P-S21}).
   \begin{theorem}  \label{kupia}
Let $T\in \ogr(\hh)$ be a class A operator and
$n, \kappa$ be integers greater than $1$. If
$X=T^n$ satisfies the single equation
\eqref{xukl-5}, then $T$ is quasinormal.
   \end{theorem}
   \begin{proof}
By assumption, we have
$(T^{*n\kappa}T^{n\kappa})^{\frac1{n\kappa}}
=(T^{*n}T^n)^{\frac1n}$. Since $n, \kappa \Ge
2$, Theorem~\ref{rowm}(i) implies that $T$ is
quasinormal.
   \end{proof}
   Clearly, Theorem~\ref{kupia} implies
Theorem~\ref{twhyp}. It is also worth noting
that Theorem~\ref{kupia} is no longer true for
$n=1$ and $\kappa=2$ (see~\eqref{qqq}). In this
particular case, the single equation
\eqref{xukl-5} automatically implies that $T$ is
of class A.
   \section{\label{Sec.5}Subnormal $n$th roots of unbounded quasinormal operators}
In this section, we will give the proof of
Theorem~\ref{twsub}. Comparing this proof with
the second proof of \cite[Theorem~1.2]{P-S21},
one can find out that the case of closed,
densely defined operators is much more
elaborate. We will start with two auxiliary
lemmas.
   \begin{lemma} \label{nkngwk}
Suppose that $N$ is a normal operator in $\hh$
and $k\in \zbb_+$. Then $(N^k)^*=N^{*k}$ and
$\ddd(N^k)=\ddd(N^{*k})$.
   \end{lemma}
   \begin{proof}
Let $E$ be the spectral measure of $N$. It
follows from \cite[Theorem~5.9]{Sch12} and the
measure transport theorem (see
\cite[Theorem~5.4.10]{Bir-Sol87}) that
   \begin{align}  \label{ngwuz}
N^* = \int_{\cbb} \bar z E(\D z) = \int_{\cbb} z
\widetilde E(\D z)
   \end{align}
and
   \begin{align*}
(N^k)^* = \Big(\int_{\cbb} z^k E(\D z)\Big)^* =
\int_{\cbb} \bar z^k E(\D z) = \int_{\cbb} z^k
\widetilde E (\D z) \overset{\eqref{ngwuz}} =
N^{*k},
   \end{align*}
where $\widetilde E \colon \borel(\cbb) \to
\ogr(\hh)$ is the spectral measure given by
$\widetilde E(\varDelta) = E(\varDelta^*)$ for~
$\varDelta \in \borel(\cbb)$. Since $N^k$ is
normal, we conclude that
$\ddd(N^k)=\ddd((N^k)^*) = \ddd(N^{*k})$.
   \end{proof}
The next lemma is due to Szafraniec (see
\cite[Fact~D]{Sza00}). For the reader's
convenience we provide its proof.
   \begin{lemma} \label{pomoc-bis}
Let $T$ be a subnormal operator in $\hh$ with
normal extension $N$ acting in $\kk$ and let
$k\in \natu$. Then
   \allowdisplaybreaks
   \begin{gather*}
P\ddd(N^{*k}) \subseteq \ddd(T^{*k}),
   \\
P N^{*k} h = T^{*k} Ph, \quad h \in
\ddd(N^{*k}),
   \end{gather*}
where $P$ is the orthogonal projection of $\kk$
onto $\hh$. Moreover, if $T^k$ is densely
defined, then
   \begin{equation} \label{akagw}
\ddd(T^{*k})\subseteq \ddd((T^{k})^*).
\end{equation}
   \end{lemma}
   \begin{proof} We proceed by induction on $k$. If
$k=1$ and $g\in \ddd(N^*)$, then
   \begin{align*}
\langle Th,Pg\rangle = \langle
Nh,g\rangle{=}\langle h,N^*g\rangle=\langle h,P
N^*g\rangle,\quad h\in \ddd(T),
   \end{align*}
which implies that $Pg\in \ddd(T^*)$ and $P
N^*g=T^*P g$.

Assume now that for an unspecified fixed $k\in
\natu$,
   \begin{align}\label{indste}
P N^{*k}\subseteq T^{*k} P.
   \end{align}
Let $g\in \ddd(N^{*(k+1)})$. Then $g\in
\ddd(N^{*k})$, so by \eqref{indste}, $Pg \in
\ddd(T^{*k})$ and thus
   \allowdisplaybreaks
   \begin{align*}
\langle Th,T^{*k}Pg\rangle&
\overset{\eqref{indste}}{=} \langle Th,P
N^{*k}g\rangle
   \\
& \hspace{1ex}= \langle Th, N^{*k}g\rangle
   \\
& \hspace{1ex}=\langle Nh,N^{*k}g\rangle
   \\
& \hspace{1ex}=\langle h,N^{*(k+1)}g\rangle
   \\
& \hspace{1ex}=\langle h,P N^{*(k+1)}g\rangle,
\quad h\in \ddd(T).
   \end{align*}
This implies that $T^{*k}Pg\in \ddd(T^*)$, or
equivalently that $Pg\in \ddd(T^{*(k+1)})$, and
$T^{*(k+1)}Pg=P N^{*(k+1)}g$. Thus $P
N^{*(k+1)}\subseteq T^{*(k+1)} P$. The inclusion
\eqref{akagw} is well known.
   \end{proof}
   \begin{proof}[Proof of Theorem~\ref{twsub}]
Let $N$ be a normal extension of $T$ acting in a
complex Hilbert space $\kk$, $G\colon
\borel(\cbb) \to \ogr(\kk)$ be the spectral
measure of $N$ and $P\in \ogr(\kk)$ be the
orthogonal projection of $\kk$ onto $\hh$.
Define the semispectral measure $\varTheta\colon
\borel(\cbb) \to \ogr(\hh)$ by
   \begin{equation} \label{vrthta}
\varTheta(\varDelta) = PG(\varDelta)|_{\hh},
\quad \varDelta \in\borel(\cbb).
   \end{equation}
It follows from \cite[Theorem~5.9]{Sch12} and
the measure transport theorem (see
\cite[Theorem~1.6.12]{Ash00}) that
   \allowdisplaybreaks
   \begin{align} \notag
\|T^kh\|^2 =\|N^kh\|^2 & =\int_{\cbb} |z|^{2k}
\langle G(\D z)h,h\rangle
   \\  \notag
&= \int_{\cbb} |z|^{2k} \langle \varTheta(\D
z)h,h\rangle
   \\ \label{putra}
& = \int_{\real_+}x^k\langle F(\D x)h,h\rangle,
\quad h\in \ddd(T^k), \; k\in \zbb_+,
   \end{align}
where $F\colon \borel(\real_+) \to \ogr(\hh)$ is
the semispectral measure defined by
   \begin{equation} \label{fntn}
F(\varDelta) =
\varTheta(\phi^{-1}(\varDelta)),\quad \varDelta
\in \borel(\real_+),
   \end{equation}
with $\phi\colon \comp\to \real_+$ given by
$\phi(z)=|z|^2$ for $z\in \comp$. By
\cite[Proposition~5.3]{Sto02}, $T^k$ is closed
for every $k\in \zbb_+$. Since $T^n$ is
quasinormal, it follows from
Theorem~\ref{embry}(iv) that there exists a
spectral measure $E_n\colon \borel(\real_+) \to
\ogr(\hh)$ such that
   \begin{equation} \label{measure}
(T^n)^{*k}(T^n)^k=\int_{\real_+}x^kE_n(\D x),
\quad k\in \zbb_+.
   \end{equation}
For $k\in \natu$, define the homeomorphism
$\psi_k\colon \rbb_+ \to \rbb_+$ by $\psi_k(x) =
x^k$ for $x\in \rbb_+$. Set
$B=\int_{\real_+}\sqrt{x}E_n(\D x)$. Then $B$ is
positive and selfadjoint. According to the
measure transport theorem, we have
   \begin{align*}
B=\int_{\real_+}\psi_2^{-1}(x)E_n(\D
x)=\int_{\real_+} x \widetilde E_n(\D x),
   \end{align*}
where $\widetilde E_n\colon \borel(\rbb_+) \to
\ogr(\hh)$ is the spectral measure defined by
   \begin{align*}
\widetilde E_n (\varDelta) =
E_n(\psi_2(\varDelta)), \quad \varDelta \in
\borel(\rbb_+).
   \end{align*}
Hence, by the spectral theorem, $\widetilde E_n$
is the spectral measure of $B$. Moreover, by
\cite[Theorem~5.9]{Sch12} and the measure
transport theorem, we have
   \begin{align}\label{dwagwi}
B^{2k}=\int_{\real_+} x^{2k}\widetilde E_n (\D
x)=\int_{\real_+}x^kE_n(\D x),\quad k\in \zbb_+.
   \end{align}
Combined with \eqref{measure}, this yields
   \begin{equation}\label{b2k}
(T^{n})^{*k}(T^{n})^{k}= B^{2k},\quad k\in
\zbb_+.
   \end{equation}

Our goal now will be to show that $F$ is a
spectral measure. For, set
   \begin{align*}
\hh_j=\mathscr{R}(\widetilde
E_n([0,j]))=\mathscr{R}(E_n([0,j^2])),\quad j\in
\natu.
   \end{align*}
Let $\mathscr{E}:=\bigcup_{j=1}^\infty \hh_j$. Since
the sequence $\{E_n([0,k])\}_{k=1}^{\infty}$ converges
to $I$ in the strong operator topology, the set
$\mathscr{E}$ is dense in $\hh$. Using the fact that
$\ddd(T^{j+1})\subseteq\ddd(T^{j})$ for all $j\in
\zbb_+$, we deduce that
   \begin{equation} \label{edinf}
\mathscr{E}\subseteq \ddd^\infty(B)
\overset{\eqref{b2k}}{\subseteq} \ddd^\infty(T).
   \end{equation}
Thus $\overline{\ddd^\infty(T)}=\hh$. It follows from
\eqref{b2k} and \eqref{edinf} that
   \begin{align} \label{edenf}
\|T^{nk}h\|^2=\langle B^{2k}h, h \rangle =
\|B^kh\|^2, \quad h\in \mathscr{E}, \; k\in
\zbb_+.
   \end{align}
By the measure transport theorem, we get
   \allowdisplaybreaks
   \begin{align} \notag
\int_{\real_+} x^{k}\langle E_n(\D x) h,h\rangle
&\overset{\eqref{dwagwi}}{=} \langle
B^{2k}h,h\rangle
   \\ \notag
&
\hspace{-.2em}\overset{\eqref{edenf}}{=}\|T^{nk}h\|^2
   \\ \notag
&\overset{\eqref{putra}}{=}\int_{\real_+}
(x^n)^{k}\langle F(\D x) h,h\rangle
    \\ \label{piec}
&\hspace{1ex}=\int_{\real_+} x^{k}\langle
\widetilde F (\D x)h, h\rangle,\quad h\in
\mathscr{E}, \; k\in \zbb_+,
   \end{align}
where $\widetilde F \colon \borel(\rbb_+) \to
\ogr(\hh)$ is the semispectral measure defined
by
   \begin{align}  \label{trifi}
\widetilde F (\varDelta) =
F(\psi_n^{-1}(\varDelta)), \quad \varDelta \in
\borel(\rbb_+).
   \end{align}
However, for any $h\in \mathscr{E}$ there exists
$j\in \natu$ such that $h \in
\hh_j=\mathscr{R}(E_n([0,j^2]))$, so
   \begin{equation*}
\int_{\real_+} x^{k}\langle E_n(\D x) h,h\rangle
= \int_{[0,j^2]} x^{k}\langle E_n(\D x)
h,h\rangle, \quad k \in \zbb_+.
   \end{equation*}
This implies that the Stieltjes moment sequence
$\{\int_{\real_+} x^{k}\langle E_n(\D x)
h,h\rangle\}_{k=0}^\infty$ is determinate for
every $h\in \mathscr{E}$. Thus, by \eqref{piec},
we have
   \begin{equation*}
\langle E_n(\varDelta) h,h\rangle = \langle
\widetilde F (\varDelta) h,h\rangle,\quad
\varDelta \in \borel(\real_+), \; h\in
\mathscr{E}.
   \end{equation*}
Since $\overline{\mathscr{E}}=\hh$, we see that
   \begin{equation}\label{eqmeas}
E_n(\varDelta) = \widetilde F(\varDelta), \quad
\varDelta \in \borel(\real_+).
   \end{equation}
Noting that the map $\borel(\real_+)\ni
\varDelta \to \psi_n^{-1}(\varDelta) \in
\borel(\real_+)$ is surjective, we deduce from
\eqref{trifi} and \eqref{eqmeas} that $F$ is a
spectral measure.

We will now show that
   \begin{equation}\label{osiem}
\ddd(J_k) = \ddd(N^{2k})\cap \hh, \quad k\in
\zbb_+,
   \end{equation}
where
   \begin{equation}  \label{dufjoka}
J_k:=\int_{\real_+} x^{k} F(\D x ),\quad k \in
\zbb_+.
   \end{equation}
For, observe that in view of the measure
transport theorem we have
   \allowdisplaybreaks
   \begin{align} \notag
\int_{\real_+} x^{2k} \langle F(\D x)h, h
\rangle & \overset{\eqref{fntn}} =\int_{\cbb}
|z|^{4k}\langle \varTheta(\D z)h,h\rangle
   \\ \label{ramka}
& \overset{\eqref{vrthta}}=\int_{\cbb}
|z|^{4k}\langle G(\D z)h,h\rangle, \quad h\in
\hh.
   \end{align}
Since $G$ is the spectral measure of $N$,
\eqref{osiem} follows from \eqref{ramka} and the
identity $N^j = \int_{\cbb} z^j G (\D z)$ which
holds for any $j \in \zbb_+$ (see
\cite[Theorem~5.9]{Sch12}).

Next, we will prove that
   \begin{equation}\label{dziewiec}
\ddd(J_k)\subseteq \ddd(T^{*k}T^k),\quad k\in
\zbb_+.
   \end{equation}
First, we show that
   \begin{equation} \label{dziesiec}
\ddd(J_k)\subseteq \ddd(T^k),\quad k\in \zbb_+.
   \end{equation}
For, note that
   \allowdisplaybreaks
   \begin{align} \notag
\mathscr{E} = \bigcup_{j=1}^\infty
\mathscr{R}(E_n([0,j^2])) & =
\bigcup_{j=1}^\infty \mathscr{R}(E_n([0,j^n]))
   \\ \notag
& =\bigcup_{j=1}^\infty
\mathscr{R}(E_n(\psi_n([0,j])))
   \\ \label{jedenascie}
& \hspace{-1.5ex}\overset{\eqref{eqmeas}}=
\bigcup_{j=1}^\infty\mathscr{R} (F([0,j])).
   \end{align}
Fix $k \in \zbb_+$ and take $h\in\ddd(J_k)$. Set
$h_j=F([0,j])h$ for $j\in \natu$. Then, by
\eqref{jedenascie}, $\{h_j\}_{j=1}^\infty
\subseteq \mathscr{E} \cap \ddd(J_k)$. Since
$h-h_j=F((j,\infty))h$ for $j\in \natu$ and
$\int_{\rbb_+} x^{2k} \langle F(\D x)h, h\rangle
< \infty$, we deduce from Lebesgue's dominated
convergence theorem that
   \begin{equation*}
\|J_k(h-h_j)\|^2=\|J_kF((j,\infty))h\|^2=\int_{(j,\infty)}
x^{2k} \langle F(\D x )h,h\rangle\rightarrow 0
\quad \text{as $j\to \infty$.}
   \end{equation*}
Hence
   \begin{align} \label{jutku}
h_j \to h \quad \text{and} \quad
J_kh_j\rightarrow J_kh \quad \text{as} \quad
j\rightarrow\infty.
   \end{align}
It follows from \eqref{edinf} and
\eqref{jedenascie} that $\{h_j\}_{j=1}^\infty
\subseteq \mathscr{E} \cap \ddd(J_k) \subseteq
\ddd^\infty(T)$, so by \eqref{putra} and
\eqref{dufjoka} we have
   \allowdisplaybreaks
   \begin{align*}
\|T^k(h_j-h_l)\|^2 & = \int_{\real_+} x^{k}
\langle F(\D x)(h_j-h_l),h_j-h_l\rangle
   \\
& = \langle J_k(h_j-h_l),(h_j-h_l) \rangle,
\quad j,l \in \natu.
   \end{align*}
Combined with \eqref{jutku}, this implies that
the sequence $\{T^kh_j\}_{j=1}^{\infty}$ is
convergent in $\hh$. Since $T^k$ is closed (see
\cite[Proposition 5.3]{Sto02}) and
$h_j\rightarrow h$ as $j\rightarrow \infty$, we
see that $h\in \ddd(T^k)$ and $T^kh_j\rightarrow
T^kh$ as $j\rightarrow \infty$. Applying
\eqref{putra}, \eqref{dufjoka} and \eqref{jutku}
again, we obtain
   \begin{align*}
\|T^kh\|^2 = \lim_{j\to \infty} \|T^kh_j\|^2 =
\lim_{j\to \infty}\langle J_kh_j,h_j\rangle =
\langle J_kh,h\rangle,
   \end{align*}
which completes the proof of \eqref{dziesiec}
and shows that
   \begin{equation}\label{dwanascie}
\|T^kh\|^2=\langle J_k h,h \rangle,\quad h\in
\ddd (J_k), \; k\in \zbb_+.
   \end{equation}

We now turn to the proof of \eqref{dziewiec}.
Fix $k \in \zbb_+$ and take $h\in \ddd(J_k)$. By
\eqref{osiem} and \eqref{dziesiec}, $h\in
\ddd(T^k) \cap \ddd(N^{2k})$ and consequently by
Lemmas~\ref{nkngwk} and \ref{pomoc-bis},
   \begin{equation*}
T^kh=N^{k}h\in \ddd(N^{k})\cap \hh =
\ddd(N^{*k})\cap \hh \subseteq \ddd (T^{*k}),
   \end{equation*}
so $h\in\ddd (T^{*k}T^k) $, which proves
\eqref{dziewiec}.

It follows from \eqref{dziewiec} and
\eqref{dwanascie} that
   \begin{align*}
\langle T^{*k}T^k h, h \rangle = \|T^kh\|^2
=\langle J_k h,h\rangle, \quad h\in \ddd(J_k),
\; k\in \zbb_+.
   \end{align*}
Since $J_k$ is densely defined, we get
   \begin{equation}  \label{jejat}
J_k \subseteq T^{*k}T^k, \quad k\in \zbb_+.
   \end{equation}
By induction, we have
   \begin{align*}
\langle T^{*k}T^k f, g \rangle = \langle f,
T^{*k}T^k g \rangle, \quad f,g \in
\ddd(T^{*k}T^k), \; k\in \zbb_+,
   \end{align*}
so $T^{*k}T^k$ is symmetric. Since $F$ is a
spectral measure, we infer from \eqref{dufjoka}
that $J_k$ is selfadjoint. By \eqref{jejat} and
maximality of selfadjoint operators, we obtain
   \begin{equation*}
T^{*k}T^k = J_k = \int_{\real_+} x^{k} F(\D x),
\quad k\in \zbb_+.
   \end{equation*}
It follows from Theorem~\ref{embry} that $T$ is
quasinormal. This completes the proof.
   \end{proof}
   \section{\label{Sec.6}Non-quasinormal $n$th roots of bounded quasinormal operators}
In this section we will discuss the question of
the existence of non-quasinormal $n$th roots of
(bounded) quasinormal operators. We begin with
the case of $n$th roots of normal operators. It
is a well-known fact that every normal operator
$T\in \ogr(\hh)$ has an $n$th root for any
integer $n\Ge 2$. Indeed, if $E$ is the spectral
measure of $T$, then $\int_{\cbb} \sqrt[n]{z}
E(\D z)$ is the $n$th root of $T$, where
$\sqrt[n]{z}$ is a Borel measurable branch of
the $n$th root on the complex plane (see e.g.,\
\cite[Proposition~1.13]{con87}). To simplify
further considerations, we will focus on square
roots of normal operators (which are complex
enough by themselves). If $\dim \hh \Ge 2$, then
there always exists a normal operator $T\in
\ogr(\hh)$ which does have a non-normal square
root. Indeed, it is enough to consider a normal
operator of the form $T=A^2 \oplus B^2 \oplus
B^2$, where $A$ and $B$ are normal operators on
complex Hilbert spaces $\mm$ and $\kk$,
respectively, and $\hh=\mm \oplus \kk\oplus \kk$
(the space $\mm$ may be absent). For, fix any
nonzero operator $C\in \ogr(\kk)$ that commutes
with $B$. Then the operator $S\in \ogr(\hh)$
defined~by
   \begin{align}  \label{bluk}
A \oplus \left[
\begin{matrix} B  & C \\ 0 & -B
\end{matrix}
\right]
   \end{align}
is a non-normal square root of $T$. It turns out
that if $\hh$ is separable and $\varkappa:=\dim
\hh \Ge 2$, then there is a normal operator
$T\in \ogr(\hh)$ that has only normal square
roots. For example, consider a compact normal
operator $T\in \ogr(\hh)$ with eigenvalues of
multiplicity $1$ (see
\cite[Theorem~7.1]{Weid80}). That $T$ does not
have a non-normal square root can be deduced
from \cite[Theorem~1]{Ra-Ros}, which states that
any square root of a normal operator is of the
form \eqref{bluk}, where $A$ and $B$ are normal
operators and $C$ is a nonzero operator that
commutes with $B$ (one of the summands in
\eqref{bluk} may be absent).

We now turn to the case of $n$th roots of
quasinormal operators. It is worth pointing out
that a quasinormal $n$th root of a normal
operator is normal (see
\cite[Theorem~5]{sta62}). It is also well known
that there are isometries that do not have
square roots (see \cite[Problems~145 and
151]{hal82}; see also \cite[p.~894]{hal70}). In
other words, quasinormal operators (even
completely non-normal) may not have square
roots. Our goal here is to show that if a
non-normal quasinormal operator has a
quasinormal $n$th root, where $n$ is an integer
grater than $1$, then it has many
non-quasinormal $n$th roots (see
Theorem~\ref{nrits} below). Clearly, by
Theorem~\ref{twhyp}, such $n$th roots are never
of class A. The proof of Theorem~\ref{nrits}
will be preceded by an auxiliary lemma.

For a given bounded sequence
$\lambdab=\{\lambda_k\}_{k=0}^{\infty}$ of
positive real numbers, there exists a unique
operator $W_{\lambdab}\in \ogr(\ell^2)$ such
that
   \begin{align*}
W_{\lambdab} e_k=\lambda_k e_{k+1}, \quad k \in
\zbb_+,
   \end{align*}
where $\{e_k\}_{k=0}^{\infty}$ is the standard
orthonormal basis of $\ell^2$; $W_{\lambdab}$ is
called a {\em unilateral weighted shift} with
weights $\lambdab$. If $\lambda_k=1$ for all
$k\in \zbb_+$, we denote the corresponding
unilateral weighted shift by $U$ and call it the
{\em unilateral shift} of multiplicity $1$.

The following lemma can be proved by straightforward
computations. We leave the details to the reader.
   \begin{lemma} \label{porws}
Let $W_{\lambdab}$ be a unilateral weighted
shift with positive weights $\lambdab =
\{\lambda_k\}_{k=0}^{\infty}$ and let $n\in
\natu$. Then the following conditions are
equivalent{\em :}
   \begin{enumerate}
   \item[(i)] $W_{\lambdab}^n=U^n$,
   \item[(ii)] $\prod_{j=0}^{n-1} \lambda_{k+j}=1$
for every $k\in \zbb_+$,
   \item[(iii)] $\prod_{j=0}^{n-1} \lambda_j=1$
and $\lambda_{kn+r} = \lambda_r$ for all $k\in \natu$
and $r=0, \ldots, n-1$.
   \end{enumerate}
   \end{lemma}
   \begin{corollary} \label{piurwa}
Let $n$ be an integer greater than $1$. Then there
exists a non-quasinormal unilateral weighted shift
$W_{\lambdab} \in \ogr(\ell^2)$ with positive weights
$\lambdab=\{\lambda_k\}_{k=0}^{\infty}$ such that
$W_{\lambdab}^n=U^n$.
   \end{corollary}
   \begin{proof}
Fix any sequence $\{\lambda_k\}_{k=0}^{n-1}$ of
positive real numbers that is not constant and
such that $\prod_{j=0}^{n-1} \lambda_j=1$.
Extend it periodically to a sequence
$\lambdab=\{\lambda_k\}_{k=0}^{\infty}$ by
setting $\lambda_{kn+r}=\lambda_r$ for $k\in
\natu$ and $r=0, \ldots, n-1$. Clearly, the
sequence $\lambdab$ is bounded. It follows from
Lemma~\ref{porws} that $W_{\lambdab}^n=U^n$.
However, $W_{\lambdab}$ is not quasinormal
because the only quasinormal unilateral weighted
shifts with positive weights are operators of
the form $t U$, where $t$ is a positive real
number.
   \end{proof}
We now show that if a non-normal quasinormal operator
$T$ has a quasinormal $n$th root with $n\Ge 2$, then
it has a non-quasinormal $n$th root. In fact, the
proof of Theorem~\ref{nrits} below gives more
information about non-quasinormal $n$th roots of
such~$T$.
   \begin{theorem} \label{nrits}
Let $T\in \ogr(\hh)$ be a non-normal quasinormal
operator and $n$ be an integer greater than $1$.
If $T$ has a quasinormal $n$th root, then it has
a non-quasinormal $n$th root.
   \end{theorem}
   \begin{proof}
Let $Q\in \ogr(\hh)$ be a quasinormal $n$th root
of $T$. According to \cite[Theorem~1]{brow53}
(see also \cite[Sec.\ II.\S3]{con91}), the
operator $Q$ takes the following form (up to
unitary equivalence)
   \begin{align} \label{swiadu}
Q = N\oplus(U\otimes S),
   \end{align}
where $N$ is a normal operator, $S$ is a positive
operator such that $\nul(S)=\{0\}$ and $U$ is the
unilateral shift of multiplicity $1$. We will consider
two cases.

{\sc Case 1.} $U\otimes S$ acts on a nonzero
complex Hilbert space.

It follows from Corollary~\ref{piurwa} that there
exists a non-quasinormal unilateral weighted shift
$W_{\lambdab} \in \ogr(\ell^2)$ with positive weights
$\lambdab=\{\lambda_n\}_{n=0}^{\infty}$ such that
$W_{\lambdab}^n=U^n$. Then we have
   \begin{align*}
\big(N\oplus(W_{\lambdab}\otimes S)\big)^n =
N^n\oplus(U^n\otimes S^n) \overset{\eqref{swiadu}}=
Q^n=T.
   \end{align*}
Therefore $R:=N\oplus(W_{\lambdab}\otimes S)$ is an
$n$th root of $T$. We show that $R$ is not
quasinormal. Indeed, otherwise $W_{\lambdab}\otimes S$
is quasinormal. Since $W_{\lambdab}$ and $S$ are
nonzero operators, it follows from
\cite[Theorem~2.4]{Sto96} that $W_{\lambdab}$ is
quasinormal, which leads to a contradiction.

{\sc Case 2.} $Q=N$.

Then $T=N^n$, which implies that $T$ is normal, a
contradiction.
   \end{proof}
In view of the discussion preceding
Lemma~\ref{porws}, the natural question arises
as to whether the converse of
Theorem~\ref{nrits} holds.
   \begin{prob}
Let $T\in \ogr(\hh)$ be a non-normal quasinormal
operator which has a non-quasinormal $n$th root,
where $n$ is an integer greater than $1$. Does
it follow that $T$ has a quasinormal $n$th root?
   \end{prob}
   \subsection*{Acknowledgements}
The authors would like to thank Professor Z. J.
Jab{\l}o\'{n}\-ski for posing the question of
the existence of non-quasinormal $n$th roots of
quasinormal operators. This problem was solved
in Section~\ref{Sec.6} in a relatively general
context (see Theorem~\ref{nrits}).
   \bibliographystyle{amsalpha}

\begin{thebibliography}{99}
\bibitem{Ando72}
T. Ando, Operators with a norm condition,
{\em Acta Sci. Math. $($Szeged\/$)$} {\bf
33} (1972), 169--178.

\bibitem{Ag85}
J. Agler, Hypercontractions and
subnormality, {\em J. Operator Theory}
{\bf 13} (1985), 203--217.

   \bibitem{Ash00}
R. B. Ash, {\em Probability and measure theory},
Harcourt/Academic Press, Burlington, 2000.

\bibitem{B-C-R}
C. Berg, J. P. R. Christensen, P. Ressel,
{\em Harmonic analysis on semigroups},
Springer-Verlag, Berlin, 1984.

\bibitem{Bir-Sol87}
M. Sh. Birman, M. Z. Solomjak, {\em Spectral
Theory of selfadjoint operators in Hilbert
space}, D. Reidel Publishing Co., Dordrecht,
1987.

\bibitem{brow53}
A. Brown, On a class of operators, {\em
Proc. Amer. Math. Soc.} {\bf 4} (1953),
723--728.

\bibitem{con91}
J. B. Conway, {\em The theory of
subnormal operators}, Math. Surveys
Monographs, Amer. Math. Soc., Providence,
1991.

\bibitem{con87}
J. B. Conway, B. B. Morrel, Roots and
logarithms of bounded operators on
Hilbert space, {\em J. Funct. Anal.} {\bf
70} (1987), 171--193.

\bibitem{Curto20}
R. E. Curto, S. H. Lee, J. Yoon,
Quasinormality of powers of commuting
pairs of bounded operators, {\em J.
Funct. Anal.} {\bf 278} (2020), 108342.

\bibitem{Dou66}
R. G. Douglas, On majorization, factorization,
and range inclusion of operators on Hilbert
space, {\em Proc. Amer. Math. Soc.} {\bf 17}
(1966), 413--415.

\bibitem{Dug93} B. P. Duggal,  On $n$th roots of normal
contractions, {\em Bull. London Math. Soc.} {\bf
25} (1993), 74--80.

\bibitem{Emb68} M. R. Embry, $n$th
roots of operators, {\em Proc. Amer. Math. Soc.}
{\bf 19} (1968), 63--68.

\bibitem{Embry73}
M. R. Embry, A generalization of the
Halmos-Bram criterion for subnormality,
{\em Acta Sci. Math. $($Szeged$)$} {\bf
35} (1973), 61--64.

\bibitem{Fur01} T. Furuta, {\em Invitation to
linear operators}, Taylor \& Francis, Ltd.,
London, 2001.

\bibitem{FIY98}
T. Furuta, M. Ito, T. Yamazaki, A
subclass of paranormal operators
including class of log-hyponormal and
several related classes, {\em Sci. Math.}
{\bf 1} (1998), 389--403.

\bibitem{Gil74} F. Gilfeather, Operator valued roots of
abelian analytic functions, {\em Pacific J.
Math.} {\bf 55} (1974), 127--148.

\bibitem{hal50}
P. R. Halmos, Normal dilations and
extensions of operators, {\em Summa
Brasil. Math.} {\bf 2} (1950), 125--134.

\bibitem{hal70}
P. R. Halmos, Ten problems in Hilbert
space, {\em Bull. Amer. Math. Soc.} {\bf
76} (1970), 887--933.

\bibitem{hal82}
P. R. Halmos, {\em A Hilbert space
problem book}, Springer-Verlag, New York
Inc., 1982.

\bibitem{hal54}
P. R. Halmos, G. Lumer, Square roots of operators II, {\em
Proc. Amer. Math. Soc.} {\bf 5} (1954), 589--595.

\bibitem{hal53}
P. R. Halmos, G. Lumer, J. J. Schaffer,
Square roots of operators, {\em Proc.
Amer. Math. Soc.} {\bf 4} (1953),
142--149.

\bibitem{He51}
E. Heinz, Beitr\"age zur St\"orungstheorie der
Spektralzerlegung, {\em Math. Ann.} {\bf 123}
(1951), 415--438.

\bibitem{Ist67} V. Istr\u{a}\c{t}escu, On some hyponormal
operators, {\em Pacific J. Math.} {\bf
22} (1967), 413--417.

\bibitem{Ito99}
M. Ito, Several properties on class A including
$p$-hyponormal and log-hyponormal operators,
{\em Math. Inequal. Appl.} {\bf 2} (1999),
569--578.

\bibitem{Ito02}
M. Ito, On classes of operators
generalizing class A and paranormality,
{\em Sci. Math. Jpn.} {\bf 7} (2002),
353--363.

\bibitem{jabl14}
Z. J. Jab{\l}o\'{n}ski, I. B. Jung, J.
Stochel, Unbounded quasinormal operators
revisited, {\em Integr. Equ. Oper.
Theory} {\bf 79} (2014), 135--149.

\bibitem{Jib10}
A. A. S. Jibril, On operators for which
$T^{*2} T^2 = (T^*T)^2$, {\em Int. Math.
Forum} {\bf 46} (2010), 2255--2262.

\bibitem{Kauf83} W. E. Kaufman, Closed operators
and pure contractions in Hilbert space, {\em
Proc. Amer. Math. Soc.} {\bf 87} (1983), 83--87.

\bibitem{Keo84} G. E. Keough, Roots of invertibly
weighted shifts with finite defect, {\em
Proc. Amer. Math. Soc.} {\bf 91} (1984),
399--404.

\bibitem{Ki-Ko22} Y. Kim, E. Ko, Characterizations
of square roots of unitary weighted composition
operators on $H^2$, {\em Complex Anal. Oper. Theory}
{\bf 16:14} (2022), 22 pp.

\bibitem{Lam76}
A. Lambert, Subnormality and weighted
shifts, {\em J. London Math. Soc.} {\bf
14} (1976), 476--480.

\bibitem{Lo34}
K. L\"owner, \"Uber monotone Matrixfunktionen,
{\em Math. Z.} {\bf 38} (1934), 177--216.

\bibitem{MPR}
J. Mashreghi, M. Ptak, W. T. Ross, Square roots
of some classical operators, arXiv:2109.13688

\bibitem{pp16}
P. Pietrzycki, The single equality $A^{
*n}A^n=(A^* A)^n$ does not imply the
quasinormality of weighted shifts on
rootless directed trees, {\em J. Math.
Anal. Appl} {\bf 435} (2016), 338--348.

\bibitem{pp18}
P. Pietrzycki, Reduced commutativity of
moduli of operators, {\em Linear Algebra
Appl.} {\bf 557} (2018), 375--402.

\bibitem{P-S21}
P. Pietrzycki, J. Stochel, Subnormal $n$th roots
of quasinormal operators are quasinormal, {\em
J. Funct. Anal.} {\bf 280} (2021), 109001.

\bibitem{P-S22} P. Pietrzycki, J. Stochel, Corrigendum to
``Subnormal $n$th roots of quasinormal operators
are quasinormal'' [J. Funct. Anal. 280 (2021)
109001], {\em J. Funct. Anal.} {\bf 282} (2022),
109260.

\bibitem{P-S22-b}
P. Pietrzycki, J. Stochel, Two-moment
characterization of spectral measures on the
real line, submitted.

\bibitem{put57}
C. R. Putnam, On square roots of normal
operators, {\em Proc. Amer. Math. Soc.}
{\bf 8} (1957), 768--769.

\bibitem{Ra-Ros} H. Radjavi, P. Rosenthal,
On roots of normal operators, {\em J. Math.
Anal. Appl} {\bf 34} (1971), 653--664.

\bibitem{Rud73}
W. Rudin, {\em Functional analysis},
McGraw-Hill Series in Higher Math.,
McGraw-Hill Book Co., New York, 1973.

\bibitem{Sch12}
K. Schm\"{u}dgen, {\em Unbounded
self-adjoint operators on Hilbert space,}
Graduate Texts in Mathematics, 265,
Springer, Dordrecht, 2012.

\bibitem{sta62}
J. G. Stampfli, Hyponormal operators, {\em
Pacific J. Math.} {\bf 12} (1962), 1453--1458.

\bibitem{sta66}
J. G. Stampfli, Which weighted shifts are
subnormal?, {\em Pacific J. Math.} {\bf 17}
(1966), 367--379.

\bibitem{Sto92}
J. Stochel, Decomposition and
disintegration of positive definite
kernels on convex $*$-semigroups, {\em
Ann. Polon. Math.} {\bf 56} (1992),
243--294.

\bibitem{Sto96}
J. Stochel, Seminormality of operators from
their tensor product, {\em Proc. Amer. Math.
Soc.} {\bf 124} (1996), 135--140.

\bibitem{Sto02}
J. Stochel, Lifting strong commutants of unbounded
subnormal operators, {\em Integr. Equ. Oper. Theory}
{\bf 43} (2002), 189--214.

\bibitem{Sto-Szaf85} J. Stochel, F. H. Szafraniec,
On normal extensions of unbounded operators. I,
{\em J. Operator Theory} {\bf 14} (1985),
31--55.
  \bibitem{Sto-Szaf89} J. Stochel, F. H. Szafraniec,
On normal extensions of unbounded operators. II,
{\em Acta Sci. Math. $($Szeged$)$}, {\bf 53}
(1989), 153--177.
   \bibitem{Sto-Szaf89III}
J. Stochel, F. H. Szafraniec, On normal
extensions of unbounded operators. III. Spectral
properties, {\em Publ. RIMS, Kyoto Univ.} {\bf
25} (1989), 105--139.
   \bibitem{Sto-Szaf98}
J. Stochel, F. H. Szafraniec, The complex moment
problem and subnormality: a polar decomposition
approach, {\em J. Funct. Anal.} {\bf 159}
(1998), 432--491.
   \bibitem{Sza00} F. H. Szafraniec, Subnormality in
the quantum harmonic oscillator, {\em Comm.
Math. Phys.} {\bf 210} (2000), 323--334.
   \bibitem{Tana99} K. Tanahashi, On log-hyponormal
operators, {\em Integr. Equ. Oper. Theory} {\bf
34} (1999), 364--372.
   \bibitem{uch93}
M. Uchiyama, Operators which have
commutative polar decompositions, {\em
Oper. Theory Adv. Appl.} {\bf 62} (1993),
197--208.
   \bibitem{uch01}
M. Uchiyama, Inequalities for semibounded
operators and their applications to
log-hyponormal operators, {\em Oper. Theory Adv.
Appl.} {\bf 127} (2001), 599--611.
   \bibitem{Weid80}
J. Weidmann, {\it Linear operators in
Hilbert spaces}, Springer-Verlag, Berlin,
Heidelberg, New York, 1980.
   \bibitem{wog85}
W. Wogen, Subnormal roots of subnormal
operators, {\em Integr. Equ. Oper.
Theory} {\bf 8} (1985), 432--436.
   \bibitem{Yama99} T. Yamazaki,
Extensions of the results on $p$-hyponormal and
$\log$-hyponormal operators by Aluthge and Wang,
{\em SUT J. Math.} {\bf 35} (1999), 139--148.
   \end{thebibliography}
   
   \end{document}